\documentclass[12pt,reqno]{amsart}
\usepackage{amsmath,amssymb,verbatim,amscd}

\theoremstyle{plain}
\newtheorem{theorem}{Theorem}[section]
\newtheorem{proposition}[theorem]{Proposition}
\newtheorem{lemma}[theorem]{Lemma}

\theoremstyle{definition}
\newtheorem{definition}[theorem]{Definition}

\newtheorem{remark}[theorem]{Remark}
\newtheorem{conjecture}[theorem]{Conjecture}

\newcommand\Q{\mathbb{Q}}

\newcommand\frh{\mathfrak{h}}
\newcommand\frg{\mathfrak{g}}
\newcommand\frb{\mathfrak{b}}

\begin{document}

\title{On Kostant's conjecture for components of $V(\rho)\otimes V(\rho)$}
\begin{abstract}
For a complex simple Lie algebra $\mathfrak{g}$ or rank $r$, let $\rho$ be the half sum of positive roots and  $P(2\rho)\subset \mathbb{R}^r$ be the convex hull of all dominant weights  $\lambda$ of the form $\lambda=2\rho-\sum_{i=1}^r a_i\alpha_i$ with $a_i\in \mathbb{Z}_{\geq 0}$ for $1\leq i\leq r$.  

We show that if $\lambda$ is a vertex of $P(2\rho)$, then $V(\lambda)$ appears in $V(\rho) \otimes V(\rho)$ with multiplicity one, proving partially (for the vertices of $P(2\rho)$) a conjecture of Kostant describing components of $V(\rho)\otimes V(\rho)$. This result allows us to give an alternative proof for a weaker form of the conjecture (up to saturation factor) for any $\mathfrak{g}$.  Further, using works of Knutson-Tau on the saturation property of $\mathfrak{sl_{r+1}}$, our results give an alternative proof of Kostant's conjecture in the particular case $\mathfrak{g}=\mathfrak{sl_{r+1}}$.
\end{abstract}
\author{Arzu Boysal}
\address{Department of Mathematics\\ Bo\u{g}azi\c{c}i University\\
34342 \\ Istanbul\\ Turkiye.}

\email{arzu.boysal@bogazici.edu.tr}

\date{}
\maketitle \pagestyle{myheadings}

\markboth{}{}  \maketitle

\section{Introduction}

Let $\frg$ be a simple complex Lie algebra of rank $r$. We fix a Borel subalgebra $\frb$ and a Cartan subalgebra $\frh\subset \frb$ of $\frg$.  Let $\{\alpha_1, \dots, \alpha_r\}$ be the set of simple roots, and $\rho$ half sum of positive roots.  
For an integral dominant weight $\lambda \in \frh^*$, let $V(\lambda)$ be the corresponding irreducible representation of $\frg$ with highest weight $\lambda$.
Denote the multiplicity of a representation $V(\nu)$ in the tensor product $V(\lambda) \otimes V(\mu)$  by $c_{\lambda \mu}^{\nu}$.  

The exterior algebra $\Lambda^*\frg$, as an adjoint $\frg$-module, is isomorphic to the direct sum of $2^r$ copies of $V(\rho)\otimes V(\rho)$ (\cite{Ko} Section 5).  The following conjecture of B.~Kostant, describing all constituents of  $V(\rho)\otimes V(\rho)$, and thus that of $\Lambda^*\frg$, appears in Berenstein-Zelevinsky \cite{BZ}.

\begin{conjecture}\label{c:conjkostant} (Kostant) Suppose $\lambda$ is a dominant integral weight of $\frg$. Then, $c_{\rho \rho}^{\lambda}> 0$ if and only if the weight $2\rho-\lambda$ is a nonnegative integral linear combination of simple roots in $\frg$.
\end{conjecture}
The necessity above is clear; if $V(\lambda)$ is a component of $V(\rho) \otimes V(\rho)$ then $2\rho-\lambda$ is a nonnegative integral linear combination of simple roots.   To our knowledge, the following is known on the conjecture to date.

In 1991, Berenstein-Zelevinsky proved the conjecture for $\frg=\mathfrak{sl_{r+1}}$, in the same article that the conjecture appeared (\cite{BZ} Theorem 6), as an application of the combinatorial method they developed for triple product multiplicities.  Further results for the multiplicities involved are given therein; they proved for $\frg=\mathfrak{sl_{r+1}}$ that the multiplicity $c_{\rho \rho}^{\lambda}=1$ if and only if $\lambda$ is a vertex of $P(2\rho)$ (\cite{BZ} Theorem 14), where $P(2\rho)$  denotes the convex hull of all dominant weights  $\lambda$ of the form $\lambda=2\rho-\sum a_i\alpha_i$ with $a_i\in \mathbb{Z}_{\geq 0}$ for $1\leq i\leq r$.

We remark that the necessity above fails for other root types; $V(\nu)$ can appear with multiplicity one  $V(\rho)\otimes V(\rho)$ for a weight $\nu$ that is not a vertex of $P(2\rho)$. This is already apparent in $B_2$, where $c_{\rho \rho}^{2\omega_2}=1$, yet $2\omega_2$ is not a vertex of $P(2\rho)$ in this case.

In 2016, Chirivi-Kumar-Maffei \cite{CKM} showed the validity of a weaker form of the conjecture, up to saturation factor (cf.~ Definition \ref{d:sfactor}).  In 2022, Nadimpalli-Pattanayak \cite{NP} obtained the same result as a special case of their branching analysis of $V(\rho)$.  More recently in 2023 Jeralds-Kumar \cite{JK} verified this weaker form for any affine Kac-Moody Lie algebra.

For simple $\frg$ of exceptional type the conjecture can be verified utilizing Lie theoretic software.

\medskip

In this manuscript we first give an elementary proof of characterization of vertices of dominant weight polytope $P(\mu)$ for any regular dominant integral weight $\mu$,  reproving results of 
Besson-Jeralds-Kiers (\cite{BJK} Theorem 3.7) and finite version of Jeralds-Kumar (\cite{JK}  Proposition 3.9, Corollary 3.10).  More precisely we prove the following.

\begin{theorem}\label{th:main3}
Suppose $\mu$ is a regular dominant integral weight of $\mathfrak{g}$.  Denote by $P(\mu)$ the polytope given by the intersection of the convex hull of the Weyl group orbit $W\cdot \mu$ and the dominant rational cone.  Then, the vertices of $P(\mu)$ are in bijective correspondence with subsets $J\subseteq \{1,2,\cdots, r\}$. A vertex $v_J$ corresponding to $J$ is equal to $$v_J=\frac{1}{|W_J|}\displaystyle \sum_{w\in W_J} w(\mu),$$ where  $W_J$ is the subgroup of the Weyl group $W$ generated by simple reflections $s_{\alpha_i}\;(i\in J)$. 
\end{theorem}

The special case of Theorem \ref{th:main3} for $\mu=2\rho$ verifies Kostant's characterization of vertices of $P(2\rho)$, which was stated in Berenstein-Zelevinsky (\cite{BZ} Proposition 13) without proof.  This characterization gives the interpretation of vertices of $P(2\rho)$ as  dominant PRV components; a vertex $v_J$ corresponding to $J$ is equal to $\rho+w_J\rho$, where $w_J$ is the maximal element in $W_J$ (cf.~ Proposition \ref{p:kostant}).  
This particular characterisation of  vertices let us  give  self contained proofs of the following two assertions related to Conjecture \ref{c:conjkostant}. The first one below extends the sufficiency part of Berenstein-Zelevinsky (\cite{BZ} Theorem 14) to other Lie algebra types (as mentioned above the necessity fails in general), thus proving Conjecture \ref{c:conjkostant} for the vertices of $P(2\rho)$.  

\begin{theorem}\label{th:main1} Let $\frg$ be any simple finite dimensional Lie algebra over complex numbers.  Suppose $\lambda$ is a dominant integral  weight of $\frg$.

If $\lambda$ is a vertex of $P(2\rho)$, then $V(\lambda)$ appears in $V(\rho) \otimes V(\rho)$.  Moreover, in this case $c_{\rho \rho}^{\lambda}=1$. 
\end{theorem}

Theorem \ref{th:main1} could be verified using PRV's desription of tensor product multiplicities \cite{PRV}.  
Here instead we opt to give a proof using Weyl character formula, and a particular form of it giving the interpretation of representations as Euler-Poincare chacateristic of the cohomology groups involved (as appeared in the original setting of the problem conjectured by Kostant). This formulation manifests the correspondence between certain integral weights in $P(2\rho)$ and the indexing set in the character formula of $V(\rho)\otimes V(\rho)$. The proof is in Section \ref{s:character}.

Using Theorem \ref{th:main1}, the convexity of $P(2\rho)$, and the semigroup property of the tensor cone, we get an elementary proof for the weaker form of Kostant's conjecture up to saturation factor, and hence yet another proof for  Conjecture \ref{c:conjkostant} for $\frg=\mathfrak{sl_{r+1}}$. 
This result  also appeared in Chirivi-Kumar-Maffei \cite{CKM} and Nadimpalli-Pattanayak \cite{NP} as mentioned above.  More precisely we show the following, its proof is in Section \ref{s:saturation}.

\begin{theorem}\label{th:main2} Let $\frg$ be any simple finite dimensional Lie algebra over complex numbers. 
If $\lambda$ is a dominant integral weight of $\frg$ satisfying $2\rho-\lambda=\sum a_i\alpha_i, \;a_i\in \mathbb{Z}_{\geq 0}$, and if $d$ is a saturation factor of $\frg$, then $V(d\lambda)$ appears in $V(d\rho) \otimes V(d\rho)$. 

In particular, for $\frg=\mathfrak{sl_{r+1}}$ Kostant's conjecture \ref{c:conjkostant} holds, as the saturation factor $d=1$ in this case.
\end{theorem}

We emphasize that the methods used in this manuscript proving Theorem \ref{th:main1} and Theorem \ref{th:main2} do not involve any technical machinery  employed in \cite{CKM} or \cite{NP} (such as Schubert calculus, Hilbert-Mumford criteria), and have no common construction with that of \cite{BZ} (their methods work only for  $\frg=\mathfrak{sl_{r+1}}$).

\section{Preliminary constructions and vertices of dominant weight polytopes}

\subsection{Preleminaries}
Let $\frg$ be a simple Lie algebra over $\mathbb C$ of rank $r$.
Let $\Phi=R(\frh,\frg)\subset \frh^*$ be
the root system.
The choice of $\frb$ determines a set of simple roots $\Delta=\{\alpha_1, \dots, \alpha_r\}$ of $\Phi$.  Let $\frh_{\mathbb{R}}$ denote the real span of elements of $\frh$
dual to $\Delta$.  For each root $\alpha$, denote by $H_{\alpha}$ the
unique element of $[\frg_{\alpha},\frg_{-\alpha}]$ such that
$\alpha(H_{\alpha})=2$; it is called the coroot associated to the
root $\alpha$.
Let $\{\omega_i \}_{1\leq i\leq r}$ be the set of fundamental
weights, defined as the basis of $\frh^*$ dual to
$\{H_{\alpha_i}\}_{1\leq i\leq r}$. 

Let $\rho$ denote the sum of fundamental weights; it is equal to the half sum of positive roots.

Denote the Weyl group of $\frg$ by $W$, it is generated by elements $s_{\alpha}$ for $\alpha \in \Delta$; acting on $\frh_{\mathbb{R}}^*$ as $s_{\alpha}\beta=\beta-\beta(H_{\alpha})\alpha.$

We define the weight lattice
$P=\{ \lambda \in \frh^* : \lambda(H_{\alpha}) \in \mathbb{Z},\;
\forall \alpha \in \Phi \}$, and denote the set of dominant weights by
$P^+$, that is,
$$P^+:=\{ \lambda \in \frh^* : \lambda(H_{\alpha_i}) \in
\mathbb{Z}_{\geq 0}\;
\forall \alpha_i \in \Delta \},$$ where $\mathbb{Z}_{\geq 0}$ is the set of
nonnegative integers.  The  set $P^+$ parametrizes the set of isomorphism
classes of all the finite dimensional irreducible representations of $\frg$.
%
For $\lambda\in P^+$, let $V(\lambda)$
be the associated  finite dimensional irreducible
representations of $\frg$ with highest weight $\lambda$. 
Let $w_0$ denote the longest element in the Weyl group $W$.  Then for any $\lambda \in P$, the dual of $\lambda$ denoted by $\lambda^*$ is $-w_0\lambda$. 


%
%

\subsection{Characterization of vertices of dominant weight polytopes}\label{s:inpoly}
Let $\mu \in P^+$. The dominant weight polytope $P(\mu)$ associated to $\mu$ (also called the intersection polytope of $\mu$) is defined as
$$P(\mu):=\text{conv}(W\cdot \mu)\cap\mathcal C^+_\Q, $$ where $C^+_{\Q}=\{\lambda\in P_{\mathbb{Q}}: \lambda(H_{\alpha_i})\geq 0,\, \forall\,
\alpha_i \in \Delta\}$ is the dominant rational chamber.  
It is immediate from theory of weights for the irreducible $\frg$-module $V(\mu)$ that dominant rational weights $\lambda=\sum_{i=1}^r m_iw_i$ lying in $P(\mu)$ are precisely the solutions of 
\begin{equation}\label{eq:dompoly}
\mu=\sum_{i=1}^r m_iw_i+\sum_{i=1}^r a_i\alpha_i
\end{equation} 
with $a_i, m_i, \in \Q_{\geq 0}$ for $1\leq i\leq r$ such that $N\lambda$ is integral weight of $V(N\mu)$ for some positive integer $N$ (see e.g. \cite{HumpL} Proposition 21.3).

We will characterize vertices of $P(\mu)$ for $\mu$ regular, that is, for $\mu$ satisfying $\mu(H_\alpha)>0$ for each $\alpha \in \Delta$.  The characterization is also given by  Besson-Jeralds-Kiers \cite{BJK} and for the more general Kac-Moody setting by Jeralds-Kumar \cite{JK} by some intricate  methods.  Here we give a short elementary proof using linear algebra, inspired by some of the arguments in Berenstein-Zelevinsky \cite{BZ}.

\begin{proposition}\label{p:vertexpoly}
Suppose $\mu$ is a regular dominant integral weight. A dominant rational weight $\lambda=\sum_{i=1}^r m_iw_i$ is a vertex of $P(\mu)$ if and only if $\min(m_i,a_i)=0$ for all $i$ in equation (\ref{eq:dompoly}), excluding the case that $a_i$ and $m_i$ simultaneously zero for any index $i$.
\end{proposition}

\begin{proof}
For any $\lambda$ of the form $\lambda=\sum_{i=1}^r m_iw_i=\mu-\sum_{i=1}^r a_i\alpha_i$ with $a_i, m_i\in \mathbb{Q}_{\geq 0}$ for $1\leq i\leq r$, it is immediate from the theory of weights that $m_i$ and $a_i$ cannot be simultaneously zero for any $i$. (If  $m_j=a_j=0$ for any index $j$, evaluate both sides at $H_{\alpha_j}$, and use the regularity of $\mu$, that is $\mu(H_{\alpha_j})>0$, and $\alpha_i(H_{\alpha_j})\leq 0$ for $i\neq j$)

Now let $J:=\{i\in\{1,\cdots, r\}: m_i=0\}$.  If $\min(m_i,a_i)=0$ for all $1\leq i\leq r$, then 
by the above remark, $\{i\in\{1,\cdots, r\}: a_i=0\}=\{1,\cdots, r\}\setminus J$.  Therefore we have
\begin{equation}\label{eq:linearsys}
\mu=\sum_{i\notin J} m_iw_i+\sum_{i\in J}a_i\alpha_i.
\end{equation}
As the linearly independent sets $\{w_i, i\notin J\}$ and $\{\alpha_i, i \in J\}$ are orthogonal, their union form a basis for $\mathbb{R}^r$. Then the linear system  (\ref{eq:linearsys}) has a unique solution, hence the solution is a vertex.  

Similarly, if $\lambda=\sum_{i=1}^r m_iw_i=\mu-\sum_{i=1}^r a_i\alpha_i$ is a vertex of $P(\mu)$, then $\min(m_i,a_i)$ is necessarily zero for all $1\leq i \leq r$, since otherwise the sets 
$I_1:=\{i\in\{1,\cdots, r\}: m_i\neq 0\}$, $I_2:=\{i\in\{1,\cdots, r\}: a_i\neq 0\}$ intersect nontrivially, and as we still have $I_1\cup I_2=\{1,\cdots, r\}$, the linear system
\begin{equation}\nonumber
\mu=\sum_{i\in I_1} m_iw_i+\sum_{i\in I_2}a_i\alpha_i
\end{equation}
will not have a unique solution. 

\end{proof}

We now give the precise form of vertices of $P(\mu)$.  Some more notation is necessary.

For a subset $J\subseteq \{1,\cdots,r\}$, let $\Delta_J:=\{\alpha_i \in \Delta:i\in J\}$, and
$\Phi_J:=\Phi\cap \{\mathbb{R}\;\text{span of}\;\Delta_J\}$.
Denote by $\Phi_J^+$ the set of positive roots of the root system relative to $\Delta_J$. 
Then, $\Phi_J^+=\Phi^+\cap \Phi_J$ and negative roots $\Phi_J^-=-\Phi_J^+$. 
 
Let $W_J$ be the subgroup of $W$ generated by simple reflection $s_\alpha$ with $\alpha \in \Delta_J$; denote by $w_J$ the longest element in $W_J$.

It is well known that roots in $\Phi^+$ sent by $w\in W_J$ to negative roots are precisely the roots in $\Phi_J^+$ sent by $w$ to negative roots (see e.g.~ \cite{HumpC}).

The following facts are crucial for what follows.

\begin{lemma}\label{l:coeffgeneral} Let $J$ be a subset of $\{1,\cdots,r\}$. Suppose $\mu$ is a regular dominant integral weight.  Then,
\begin{itemize}
\item[(a)] $\displaystyle \sum_{w\in W_J} w(\mu)$ is dominant integral, and if we express 
$$\displaystyle \sum_{w\in W_J} w(\mu)=\sum_{j=1}^r m_jw_j,$$  then $m_j=0$ for $j\in J$.
\item[(b)] $\displaystyle \sum_{w\in W_J} w(\mu)=|W_J|\mu-\sum_{i\in J} b_i\alpha_i$ with $b_i$ nonnegative integers.
\end{itemize}
\end{lemma}

\begin{proof}
Part $(a)$ follows from the following computation.  For a simple root $\alpha_i$, if $i\in J$, then 
$$\displaystyle \sum_{w\in W_J} w(\mu)(H_{\alpha_i})=\displaystyle \sum_{w\in W_J} \mu(H_{w^{-1}\alpha_i})=\mu(\sum_{w\in W_J}H_{w^{-1}\alpha_i})=0,$$ as the sum of all coroots of $\Phi_J$ is zero. 
As for dominance, for any $w\in W_J$ and $i\notin J$, $w(\alpha_i)$ is some positive root in $\Phi^+\setminus \Phi_J^+$.  Then we have that 
$$\displaystyle \sum_{w\in W_J} w(\mu)(H_{\alpha_i})=\displaystyle \sum_{w\in W_J}\mu(H_{w^{-1}\alpha_i})>0$$ as each $H_{w^{-1}\alpha_i}$ is a positive coroot. Integrality is immediate.

Part $(b)$ follows immediately from the trivial equality $$\displaystyle \sum_{w\in W_J}(\mu- w(\mu))=|W_J|\mu-\displaystyle \sum_{w\in W_J}w(\mu),$$  as the left handside of the above equation is a nonnegative integral linear combination of roots in $\Delta^+_J$. 
\end{proof}


Now the claim in Theorem \ref{th:main3} is immediate.
\subsubsection{Proof of Theorem \ref{th:main3}}
\begin{proof}  By part $(b)$ of Lemma \ref{l:coeffgeneral},
$$\frac{1}{|W_J|}\displaystyle \sum_{w\in W_J}(\mu- w(\mu))=\mu-\frac{1}{|W_J|}\displaystyle \sum_{w\in W_J}w(\mu)=\displaystyle\sum_{i\in J}a_i\alpha_i$$
for some $a_i\in \Q_{\geq 0}$. Further, by part $(a)$ of Lemma \ref{l:coeffgeneral}, $\frac{1}{|W_J|}\displaystyle\sum_{w\in W_J} w(\mu)$ is of the form $$\frac{1}{|W_J|}\displaystyle\sum_{w\in W_J} w(\mu)=\sum_{i\notin J} m_iw_i$$ with $m_i\in \Q_{\geq 0}$.
Hence, it follows immediately from Proposition \ref{p:vertexpoly} that $\frac{1}{|W_J|}\sum_{w\in W_J} w(\mu)$ is a vertex of $P(\mu)$, and any vertex of $P(\mu)$ is of this form for some subset $J\subseteq \{1,2,\cdots, r\}$.
\end{proof}

The following special case of Theorem \ref{th:main3} when $\mu=2\rho$ reveals a particular characterization of vertices of $P(2\rho)$ which is attributed to Kostant in Berenstein-Zelevinsky (\cite{BZ} Proposition 13).

\begin{proposition}(Kostant).\label{p:kostant}
Vertices of $P(2\rho$) are in bijective correspondence with subsets $J\subseteq \{1,2,\cdots, r\}$. A vertex $v_J$ corresponding to $J$ is equal to $\rho+w_J\rho$, where $w_J$ is the maximal element in the Weyl group generated by reflections $s_{\alpha_i}\;(i\in J)$.
\end{proposition}

\begin{proof} Using Theorem \ref{th:main3} with $\mu=2\rho$, we only need to show $$\sum_{w\in W_J} w(2\rho)=|W_J|(\rho+w_J\rho),$$ where $w_J$ is the maximal element in $W_J$.  This identity could be verified simply by evaluating both sides at each simple coroot $H_\alpha$, $\alpha \in \Delta$.  Alternatively, observe that, as $w_J$ permutes the set $\Phi^+\setminus \Phi_J^+$, the integral  weight $\rho+w_J\rho$ is dominant and when expressed as $\sum_{j=1}^r m_jw_j$, $m_j=0$ for $j\in J$.  It is also immediate that $\rho+w_J\rho=2\rho-\displaystyle\sum_{\alpha\in \Phi_J^+}\alpha.$  Thus, for each $J$, $\rho+w_J\rho$ satisfies conditions in Proposition \ref{p:vertexpoly} for being a vertex.
\end{proof}


We further remark that by characterization in Proposition \ref{p:kostant} the vertices of $P(2\rho)$ are dominant \textit{integral} weights.

\section{Character of $V(\rho)\otimes V(\rho)$, its relation to $P(2\rho)$ and proof of Theorem \ref{th:main1}}\label{s:character}

Consider the weight space decomposition for a representation $V(\mu)$, $$V(\mu)=\displaystyle \sum_{\nu\in \frh^*}m_\mu(\nu)V_\nu$$ where $V_\nu=\{v \in V(\mu): h\cdot v=\nu (h)v\;\text{for all}\; h\in \frh\}$ and $m_\mu(\nu)$ is the dimension of the weight space $V_\nu$.
Let $\Pi(\mu)$ denote the set of all weights of $V(\mu)$.

Let $\mathbb{Z}[P]$ be the group ring of $P$ with canonical basis $\{e^\mu, \mu \in P\}$.
Consider the character homomorphism $\chi: \mathcal{R}(\frg) \to
\mathbb{Z}[P]$ mapping $V \mapsto \sum_{\mu\in P}(\dim(V_{\mu}))e^{\mu}$.
For $\lambda$ dominant integral, denote $\chi(V(\lambda))$ in short by $\chi_\lambda$; it is given by
the Weyl character formula

$$\chi_{\lambda}=\frac{\sum_{w \in W}\varepsilon(w) e^{w(\lambda+\rho)}}
{\sum_{w \in W} \varepsilon(w) e^{w\rho}}.$$
Denote $D:=\sum_{w \in W} \varepsilon(w) e^{w\rho}$.  Then, using the fact that  the elements of $\Pi(\mu)$ are permuted by $W$ and $\text{dim} V_\nu=\text{dim}V_{w \nu}$ for any $w\in W$, 
\[
\begin{split}
\chi_\rho \chi_\rho &=\Big(\displaystyle \sum_{\beta \in \Pi(\rho)} m_\rho(\beta) e^\beta \Big) D^{-1}\sum_{w \in W} \varepsilon(w) e^{w(2\rho)}\\
& = D^{-1} \displaystyle \sum_{w \in W}\sum_{w\beta \in \Pi(\rho)}m_\rho(w\beta)e^{w\beta} \varepsilon(w) e^{w(2\rho)}\\
& = D^{-1} \displaystyle \sum_{w \in W}\sum_{\beta \in \Pi(\rho)}m_\rho(\beta) \varepsilon(w) e^{w(\beta+2\rho)},
\end{split}
\]
hence we get,
\begin{equation}\label{eq:decomp1}
\chi_\rho \chi_\rho =\displaystyle\sum_{\beta \in \Pi(\rho)}m_\rho(\beta) D^{-1}  \sum_{w \in W} \varepsilon(w) e^{w(\beta+2\rho)}.
\end{equation}

We now make the following observations on the configuration of weights appearing in formula (\ref{eq:decomp1}) above.  

If a weight $\beta+2\rho$ is not regular (that is if it is fixed by a Weyl group element) then the second summation in equation (\ref{eq:decomp1})  vanishes as it is W anti-invariant, and thus such a weight has no contribution in the decomposition.  

A weight $\beta$ in $\Pi(\rho)$ can be expressed  as $\beta=\rho-\sum_{\alpha_i\in\Delta} n_i \alpha_i$ with $n_i$ nonnegative integers.  Moreover, $\beta+\rho$, which lies in the root lattice, is a nonnegative integral linear combination of simple roots (as $-\rho$ is the lowest weight of $V(\rho)$).  Therefore, weights of the form $\beta +2\rho,\; \beta \in \Pi(\rho),$ lie in the `positive root cone' 
$C_Q^+:=\{\sum a_i\alpha_i: a_i\in \mathbb{R}_{> 0}\}$.  As the inverse of the Cartan  matrix associated to any simple finite dimensional Lie algebra $\frg$  has positive (rational) entries (see e.g.~\cite{K}),  $C_Q^+$ contains the dominant chamber $C^+:=\{\sum b_i\omega_i: b_i\in \mathbb{R}_{\geq 0}\}$.
In particular, a regular weight of the form $\beta+2\rho$ with $\beta\in\Pi(\rho)$ is either strictly dominant or lie in $C_Q^+\setminus C^+$.  

Now we regroup the terms in equation (\ref{eq:decomp1}) as 

\begin{equation}
\begin{split}
\chi_\rho \chi_\rho =&\displaystyle
\sum_{\substack{\{\beta \in \Pi(\rho):\\ \beta+\rho \in C^+\}}} m_\rho(\beta) D^{-1}  \sum_{w \in W} \varepsilon(w) e^{w(\beta+2\rho)}\\
&+\sum_{\substack{\{\beta' \in \Pi(\rho):\\ \beta'+2\rho \in C_Q^+\setminus C^+\}}} m_\rho(\beta') D^{-1}  \sum_{w \in W} \varepsilon(w) e^{w(\beta'+2\rho)}.
\end{split}
\end{equation}

Now for $\beta' \in \Pi(\rho)$ with $\beta'+2\rho \in C_Q^+\setminus C^+$ let $w_{\beta'}$ be the unique element of $W$ such that $\overline{\beta'+2\rho}:=w_{\beta'}(\beta'+2\rho)$ is its orbit representative (regular) in $C^+$, then by Weyl character formula again,
\begin{equation}\label{eq:decomp3}
\chi_\rho \chi_\rho =\displaystyle
\sum_{\substack{\{\beta \in \Pi(\rho):\\ \beta+\rho \in C^+\}}} m_\rho(\beta) \chi_{\beta+\rho}
+\sum_{\substack{\{\beta' \in \Pi(\rho):\\ \beta'+2\rho \in C_Q^+\setminus C^+\}}} \varepsilon(w_{\beta'})m_\rho(\beta')\chi_{w_{\beta'}(\beta'+2\rho)-\rho}.
\end{equation}

Above is a special case of a formula for tensor product multiplicities due to Steinberg  \cite{S} (this formulation is also attributed to Brauer and Klyimk independently, see \cite{HumpL}).

We will get back to the formulation in (\ref{eq:decomp3}), after we relate the sets that we are summing over above to $P(2\rho)$.

\subsection{Characterization of  weights in $P(2\rho)$}
 
Now we want to elucidate the form of other integral weights lying in $P(2\rho)$ (not just the  vertices). 

Suppose $\mu$ is a dominant integral weight  of the form  $\mu=2\rho-\sum_{i=1}^r a_i \alpha_i$  with each $a_i\in\mathbb{Z}_{\geq 0}$, by definition $\mu$ lies in $P(2\rho)$.   By Proposition \ref{p:kostant},  vertices of $P(2\rho)$  are of the form $v_J=w_J\rho+\rho$.  This particular form of the vertices immediately implies  $v_J-\rho$ is in $\Pi(\rho)$.  Now, as $P(2\rho)$ is convex, the translated set $\{\mu-\rho, \mu \in P(2\rho)\}$ lies in the convex hull of $\Pi(\rho)$. Consequently, we get that $\mu-\rho=\rho-\sum_{i=1}^r a_i \alpha_i$ is in $\Pi(\rho)$ (not just the convex hull) as $\Pi(\rho)$ is saturated (see e.g. \cite{HumpL} sections 13.4 and 21.3), call this weight $\beta$.  Consequently we get that such a $\mu$ is of the form $\mu=\rho+\beta$ for some $\beta \in \Pi(\rho)$. 
This verifies the following correspondence.

\begin{lemma}\label{l:setdescr}
The set of $\beta \in \Pi(\rho)$ with $\beta+\rho \in C^+$ is in bijection with the set of dominant integral weights $\mu$ of the form $\mu=2\rho-\sum_{i=1}^r a_i \alpha_i$ with each $a_i\in \mathbb{Z}_{\geq 0}$, where $\beta \in \Pi(\rho)$ is identified with $\beta+\rho \in P(2\rho)$.
\end{lemma}


Now back to the decomposition formula given in Equation (\ref{eq:decomp3}).  

As is well known \cite{Ko}, components of $V(\rho)\otimes V(\rho)$ are of the form $V(\mu)$ with $\mu=\beta+\rho$ for some $\beta \in \Pi(\rho)$.  Using this and semi-simplicity of $\frg$-modules, we get from equation (\ref{eq:decomp3}) that
\begin{equation}\label{eq:decomp4}
\chi_\rho \chi_\rho =\displaystyle
\sum_{\substack{\{\beta \in \Pi(\rho):\\ \beta+\rho \in C^+\}}} \Big(m_\rho(\beta)
+\sum_{\substack{\{\beta' \in \Pi(\rho):\\ \beta'+2\rho \in C_Q^+\setminus C^+\\ \overline{\beta'+2\rho}=\beta+2\rho\}}} \varepsilon(w_{\beta'})m_\rho(\beta')\Big)\chi_{\beta+\rho}.
\end{equation}

\begin{remark}
We remark that for $\beta \in \Pi(\rho)$ satisfying $\beta+\rho \in C^+$, the  coefficient of $\chi_{\beta+\rho}$ in equation (\ref{eq:decomp4}) is possibly zero.  Showing that all such coefficients are nonzero is equivalent to Conjecture \ref{c:conjkostant} of Kostant in view of the characterization given in Lemma \ref{l:setdescr}.  We further remark that contribution from the second sum is only negative as  structure coefficients satisfy $c_{\rho \rho}^{\beta+\rho}\leq m_\rho(\beta)$.
\end{remark}

\subsection{The proof of Theorem \ref{th:main1}}
Now we look at the special case when $\beta\in \Pi(\rho)$ is of the form $w_J\rho$ for any subset $J$ of $\{1,2,\cdots, r\}$. 

We will verify that $V(w_J\rho+\rho)$  is a component of $V(\rho)\otimes V(\rho)$ by exhibiting that the second sum in equation (\ref{eq:decomp4}) is an empty sum, and thus vanishes.  This is in fact a special case of a result in Kac-Wakimoto (\cite{KW} Lemma 8.1b), as $w_J\rho+\rho$ is dominant. The  proof can also be given using PRV's  description of tensor product multiplicities \cite{PRV} as opposed to using  Weyl character formula as we do here.

Let $(,)$ be the standard $W$-invariant bilinear form on $\frg$.  As is well known (see e.g. \cite{K} Lemma 11.4), any weight $\mu \in \Pi(\rho)$ satisfies 
\begin{equation}\label{in:weights}
(\rho,\rho)-(\mu,\mu)\geq 0, 
\end{equation}
with equality if and only if $\mu$ is in the $W$-orbit of $\rho$.  Consider the weight $w_J\rho+2\rho-w(2\rho)$ for any $w\neq 1 \in W$.  Then

\begin{equation}
\begin{split}
(w_J\rho+2\rho-w(2\rho),w_J\rho+2\rho-w(2\rho))-(\rho,\rho)
=(w_J\rho+2\rho,w_J\rho+2\rho)\\+4(\rho,\rho)
+2(w_J\rho+2\rho, 2\rho-w(2\rho))-2(w_J\rho+2\rho,2\rho)-(\rho,\rho)
\end{split}
\end{equation}
As $\rho-w(\rho)$ is a sum of certain positive roots for any $w\neq 1\in W$ and $w_J\rho+2\rho$ is strictly dominant, we have that  $$(w_J\rho+2\rho, 2\rho-w(2\rho))>0.$$ Therefore,
\begin{equation}\label{i:notweight}
\begin{split}
(w_J\rho+2\rho-w(2\rho),w_J\rho+2\rho-w(2\rho))-(\rho,\rho)
>(w_J\rho+2\rho,w_J\rho+2\rho)\\+4(\rho,\rho)
-2(w_J\rho+2\rho,2\rho)-(\rho,\rho)\\
=(w_J\rho+2\rho-2\rho,w_J\rho+2\rho-2\rho)-(\rho,\rho)=0.
\end{split}
\end{equation}
Consequently, for any $w\neq 1 \in W$, $w_J\rho+2\rho-w(2\rho)$ is not in $\Pi(\rho)$, as  the inequality (\ref{in:weights}) is violated. This implies that the second summand in equation (\ref{eq:decomp4}) is zero, as there is no weight $\beta'\in \Pi(\rho)$ of the form $\beta'=w_J\rho+2\rho-w(2\rho)$ for any $w\neq 1\in W$.  Then, by equation (\ref{eq:decomp4}) again,  $c_{\rho \rho}^{w_J\rho+\rho}=1$ as $m_\rho(w_J\rho)=m_\rho(\rho)=1$, proving the claim in Theorem \ref{th:main1}. 
This establishes the validity of Conjecture \ref{c:conjkostant} for the vertex set.

\section{Saturation and the proof of Theorem \ref{th:main2}}\label{s:saturation}

Let us recall the following definition.  

\begin{definition}\label{d:sfactor}
A a positive integer $d$ is called a \textit{saturation factor} for $\frg$ if for any dominant integral weights $\lambda$, $\mu$, and $\nu$ such that $\lambda+\mu+\nu$ lies in the root lattice of $\frg$, and $(V(N\lambda)\otimes V(N\mu)\otimes V(N\nu))^\frg\neq 0$ for some integer $N>0$, then $(V(d\lambda)\otimes V(d\mu)\otimes V(d\nu))^\frg\neq 0$.  
\end{definition}
Clearly if $d$ is a saturation factor then so is its any integer multiple.  

By Knutson-Tau  \cite{KT}, $d = 1$ for $\frg=\mathfrak{sl_{r+1}}$.  By works of Kapovich-Millson \cite{KM}, improved in classical types by Belkale-Kumar \cite{BK} and Sam \cite{Sa}, for $\frg$ of types $B_r$ and $C_r$, $d=2$, and for $D_r$, $d=4$.  These results in literature, determining the smallest integer $d$, are referred as saturation theorems. 

%
\subsubsection*{Proof of Theorem \ref{th:main2}}

As $P(2\rho)$ is convex and $2\rho$ itself in the root lattice of $\frg$, any rational convex combination of vertices 
$$\dfrac{1}{\sum b_J}\sum b_J v_J,\;b_J\in \mathbb{Z}_{\geq 0}$$ lying in the root lattice of $\frg$ is in one to one correspondence with dominant integral weights of the form $2\rho-\sum_{i=1}^r a_i \alpha_i$ with each $a_i\in\mathbb{Z}_{\geq 0}$. 

Now we use the semigroup property of the tensor cone under addition (see e.g.~\cite{KLM}). Denote $\sum b_J$ by $N$. Then,  
as $V(v_J)\subset V(\rho)\otimes V(\rho)$ by Theorem \ref{th:main1}, we have
$$V(\sum b_J v_J )=V(N(2\rho-\sum_{i=1}^r a_i \alpha_i))\subset V(N\rho)\otimes V(N\rho).$$  
By saturation principle, 
$$V(d(2\rho-\sum_{i=1}^r a_i \alpha_i))\subset V(d\rho)\otimes V(d\rho),$$  where $d$ is the saturation factor of $\frg$. Hence the theorem. 

As $d=1$ for $\frg=\mathfrak{sl_{r+1}}$, Kostant's Conjecture \ref{c:conjkostant} holds in this case.

\subsection*{Acknowldegement}
I thank Prakash Belkale for his generous feedback on this manuscript.

\end{document}